 \documentclass[11pt]{amsart}

\usepackage[legalpaper, margin=1.2in]{geometry}

\usepackage{amssymb,amsfonts,amsmath,dsfont,mathrsfs}

\usepackage{mathtools}
\usepackage[dvipsnames]{xcolor}
\usepackage{soul}
\usepackage{mathrsfs}
\usepackage{latexsym}
\usepackage{stmaryrd}

\usepackage{tikz-cd}
\usetikzlibrary{matrix,arrows}

\usepackage{comment}
\usepackage{hyperref}

\newtheorem{thm}{Theorem}[subsection]
\newtheorem{prop}[thm]{Proposition}
\newtheorem{lem}[thm]{Lemma}

\newtheorem{sublem}{Sublemma}

\theoremstyle{definition}
\newtheorem{definition}[thm]{Definition}
\newtheorem{example}[thm]{Example}
\newtheorem*{question*}{Question}

\theoremstyle{remark}

\newtheorem{remark}[thm]{Remark}
\newtheorem{assumption}[thm]{Assumption}

\numberwithin{equation}{subsection}

\newcommand{\Q}{\mathbf{Q}}
\newcommand{\Z}{{\mathbf{Z}}}

\newcommand{\F}{\mathbf{F}}
\newcommand{\kfield}{{\mathbf{k}}}

\newcommand{\Gal}{{\rm Gal}}
\newcommand{\Frob}{{\rm Frob}}
\newcommand{\ab}{{\rm ab}}
\newcommand{\GQp}{G_{\mathbf{Q}_p}}

\newcommand{\GL}{{\rm GL}}
\newcommand{\PGL}{{\rm PGL}}

\newcommand{\ad}{{\rm ad\,}}
\newcommand{\Ind}{{\rm Ind}}
\newcommand{\sgn}{{\rm sgn\,}}
\newcommand{\diag}{{\rm diag}}

\newcommand{\res}{{\rm Res}}
\newcommand{\Aut}{{\rm Aut}}

\newcommand{\Hom}{{\rm Hom}}
\newcommand{\End}{{\rm End}}

\newcommand{\loc}{{\rm loc}}

\newcommand{\Spec}{{\rm Spec}}

\newcommand*{\sheafhom}{\mathcal{H}\kern -.5pt om}

\newcommand{\Cl}{{\rm Cl}}

\begin{document}




\title[\resizebox{5.5in}{!}{locally induced Galois representations with exceptional residual images}]{locally induced Galois representations with exceptional residual images}


\author{Chengyang Bao}
\address{Department of Mathematics, The University of Chicago, 
	Chicago, IL 60637}
\email{c.y.bao@math.uchicago.edu}





\begin{abstract}
In this paper, we classify all continuous Galois representations $\rho:\mathrm{Gal}(\overline{\mathbf{Q}}/\mathbf{Q})\to \mathrm{GL}_2(\overline{\mathbf{Q}}_p)$ which are unramified outside $\{p,\infty\}$ and locally induced at $p$, under the assumption that $\overline{\rho}$ is exceptional, that is, has image of order prime to $p$. We prove two results. If $f$ is a level one cuspidal eigenform and one of the $p$-adic Galois representations $\rho_f$ associated to $f$ has exceptional residual image, then $\rho_f$ is not locally induced and $a_p(f)\neq 0$. If $\rho$ is locally induced at $p$ and with exceptional residual image, and furthermore certain subfields of the fixed field of the kernel of $\overline{\rho}$ are assumed to have class numbers prime to $p$, then $\rho$ has finite image up to a twist.
\end{abstract}


\maketitle



\section{Introduction}
Let~$p$ be a rational prime. This paper studies the following general question:

\begin{question*} Determine all continuous irreducible representations
	$$\rho: \Gal(\overline{\Q}/\Q) \rightarrow \GL_2(\overline{\Q}_p)$$
	unramified outside $p$ and $\infty$
	which are locally induced at $p$ (see section \ref{LocallyInduced} for the precise definition.)
\end{question*}

If one assumes that~$\rho$ is additionally both odd and crystalline at $p$,
then $\rho$ is associated up to a twist to a classical modular form $f$ of level one with $a_p(f)=0$.
Conversely, if $a_p(f)=0$ for a classical modular form $f$ of level one, then all the $p$-adic representations $\rho_{f,p}$ attached to $f$
have this property. Hence, this question generalizes Lehmer's
question about the vanishing of the coefficients of Ramanujan's~$\Delta$ function.
In~\cite{CalegariSardari2021vanishing}, this problem was studied from the perspective of Galois deformation
theory, where it was shown that there are only finitely many
such modular representations of level one for a given $p$. 
In this paper, we study the more general question above under the additional hypothesis
that the residual semisimple representation $\overline{\rho}$ is \emph{exceptional}, which means that $p$ is coprime to the order of the image of $\overline{\rho}$. This case is interesting for two reasons.
First, it is related to the first known example (see Example \ref{Example}) of an eigenform $f$ with $a_{p}(f) \equiv 0 \bmod p$ for $p=59$ and the unique cuspidal eigenform of weight $16$. Its residual representation $\overline{\rho}$ has projective image $S_4$, which is exceptional. Second, assuming that $\overline{\rho}$ is locally induced, there \emph{does} exist a characteristic zero deformation $\rho$ of $\overline{\rho}$ which is locally induced but not globally induced; namely, the lifts $\rho$ which have finite image up to twist.
We prove two results.

\begin{thm}\label{Theorem}
	Let $f = \sum_{n \ge 1} a_n q^n$ be a cuspidal eigenform of level one with coefficients
	in $E = \mathbf{Q}(a_1, a_2, ... )$. Suppose that for some embedding $E\hookrightarrow \overline{\Q}_p,$ the $p$-adic Galois representation $\rho_f$ associated to $f$ has exceptional image. Then $a_p(f) \ne 0$ and none of the $p$-adic Galois representations attached to $f$ are locally induced.
\end{thm}

We say the residual representation $\overline{\rho}$ is {\it nicely exceptional} if it is exceptional and furthermore the class number of the fixed field 
of the projective representation ${\rm P}\overline{\rho}$ is prime to $p$. 
(We actually give a less restrictive definition of nicely exceptional representations in Definition \ref{NicelyExceptional}.)  For $\rho$ that are unramified outside $\{p,\infty\}$, which includes the case where $\rho$ is even, we have the following result.

\begin{thm}\label{ThmFinite}
	Suppose that $\overline{\rho}$ is unramified outside $\{p,\infty\}$, locally induced at $p$ and nicely exceptional (in the sense of Definition \ref{NicelyExceptional}); then every lift $\rho$ of $\overline{\rho}$ which is unramified outside $\{p,\infty\}$ and locally induced at $p$ has finite image, up to a twist. 
\end{thm}

\begin{example}\label{Example}
	Let $\overline{\rho}$ be the unique  mod-$59$ semisimple residual representation attached to the weight 16 and level one cuspidal eigenform $f$. Then $\overline{\rho}$ is exceptional with ${\rm P}\overline{\rho}(\Gal(\overline{\Q}/\Q))\cong S_4$, which was first recognized in \cite[\S4]{swinnerton1973l} and first proved in \cite{haberland1983perioden} (see also \cite{KIMING2005236}). Since the order of $\overline{\rho}(\Gal(\overline{\Q}/\Q))$ is prime to $p$ by Proposition \ref{ExceptionalResidualImage}, we obtain a lift $\rho$ of $\overline{\rho}$ to $\GL_2(\overline{\Q}_p)$ by simply lifting its finite image. One can verify with magma \cite{magma} that the class number of the $S_4$ extension is one and thus $\overline{\rho}$ is nicely exceptional. It then follows from Theorem \ref{ThmFinite} that any other lift of $\overline{\rho}$ into $\GL_2(\overline{\Q}_p)$ that is unramified outside $\{p,\infty\}$ and locally induced at $p$ is a twist of $\rho$.  
\end{example}

\begin{remark}
	In some literature, $\overline{\rho}$ is said to be exceptional if $\overline{\rho}$ is reducible or $p$ is coprime to the order of the image of $\overline{\rho}.$ We prove in Lemma \ref{AbsIrreducible} that $\overline{\rho}$ is absolutely irreducible if it is locally induced. Therefore, we adopt the more restrictive definition of exceptional representations. 
\end{remark}

\subsection*{Sketch of the proof}  To prove Theorem \ref{Theorem}, by local arguments using Breuil's theorem \cite[Prop 3.1.2]{breuil_2003II}, we show that $p$ is not congruent to $7\pmod 8$ if one of the reductions of ${\rho}_f$ is exceptional and $a_p(f)=0$. But Hatada's main theorem in \cite{Hatada1979} rules out this case.  For Theorem 1.0.3, following Calegari and Sadari's strategy, we study the deformation ring parametrizing deformations of $\overline{\rho}$ with properties in the theorem. We show that the tangent space of the deformation ring is trivial. Thus the deformation ring is the Witt ring of the residual field. Hence, it contains at most one $\overline{\Q}_p$-point and there is one $\overline{\Q}_p$-point that corresponds to the deformation of finite image up to twist. 

The paper is organized as follows. In \S2 we collect some results from representation theory and discuss the consequences of $\rho$ being locally induced and $\overline{\rho}$ being exceptional. Along the way we prove Theorem \ref{Theorem}. In \S3 we associate to $\overline{\rho}$ a deformation functor and study its tangent space under the assumptions that $\overline{\rho}$ is locally induced and exceptional. Theorem \ref{ThmFinite} is proved in \S3.

\subsection*{Notation and Conventions} 
If $F$ is a field, we denote by $G_F$ its absolute Galois group $\Gal(\overline{F}/F)$. If $T$ is a finite extension of $\Q_p,$ we let $\mathcal{O}_T$ be the ring of integers of $T$ and $\kfield_T$ be the residue field of $T.$ We let $\Frob$ denote the Frobenius map $x\mapsto x^p$ for $x\in \kfield_T.$ We fix the $p$-adic valuation $v_p:\overline{\Q}_p\to \Z \cup\{\infty\}$ normalized so that $v_p(p)=1.$ If $E$ is a number field and $w$ is a place of $E,$ we denote by $E_w$ the completion of $E$ at $w$. We let $G_w$ be the absolute Galois group of $E_w$, which is isomorphic to the decomposition group of a place of $\overline{\Q}$ over $w$. Denote by $I_w$ the inertia subgroup of $G_w$. Let $\Cl(E)$ be the class group of $E.$ If $\kfield$ is a finite field of characteristic $p,$ we denote by $W(\kfield)$ the ring of Witt vectors of $\kfield.$

If $\rho:G\to \GL_n(F)$ is a representation, then we define ${\rm P}\rho$ to be the projective representation given by $${\rm P}\rho:G\to \GL_n(F)\to {\rm P}\GL_n(F).$$ All representations are assumed to be continuous. When there is no topology mentioned, we take the discrete topology. 

For an odd prime $p,$ we let $p^*=(-1)^{(p-1)/2}p.$ All restriction maps are denoted by $\res$ unless otherwise noted in this paper.

\section{Galois Representations}

\subsection{Preliminaries}{\label{Prelim}}
In this subsection, we collect some general results from representation theory that will be used in the sequel. 

Let $\kfield$ be a field of characteristic not equal to 2 and let $\rho : G\to \GL_2(\kfield)$ be a representation of a group $G$ over $\kfield.$ Suppose that $G$ has an index two subgroup $H$. Let $\alpha$ and $\alpha'$ be two characters of $H$ into $\kfield^\times$ that are permuted by the action of $\eta:G\to G/H\cong\{\pm1\}\subseteq \kfield^\times$. Calegari and Sadari proved in \cite{CalegariSardari2021vanishing} a few lemmas on two-dimensional induced representations which we record as follows. 

\begin{lem}\label{InducedReducibleAndSemisimple}
	Suppose that $\rho\cong\Ind^G_H\alpha.$ Then $\rho$ is reducible if and only if $\alpha$ extends to a character of $G.$ In this case, $\rho\cong \alpha\oplus(\alpha\otimes\eta).$ Hence, the induced representation $\rho$ is always semisimple.
\end{lem}

\begin{proof}
	See \cite[Lemma 1.3.1]{CalegariSardari2021vanishing}. 
\end{proof}

\begin{lem}\label{InducedProjectiveDihedral}
	Suppose that the image of $\rho$ is finite. Then the following are equivalent.
	\begin{enumerate}
		\item The representation $\rho$ is irreducible and induced from some character $\alpha$ of $H$.
		\item The finite group ${\rm P}\rho(G)$ is a non-cyclic dihedral subgroup of order $2d$ with ${\rm P}\rho(H)$ cyclic of order $d$. 
	\end{enumerate} 
\end{lem}

\begin{proof}
	See \cite[Lemma 1.3.2]{CalegariSardari2021vanishing}. 
\end{proof}

\begin{lem}\label{AdjointRepDecompWhenInduced}
	 Suppose that $\rho\cong \Ind^G_H\alpha.$ Then the trace-zero adjoint representation $\ad^0\rho$ splits as $\eta\oplus \Ind^G_H ({\alpha'}^{-1}\alpha)$.
\end{lem}

\begin{proof}
	The adjoint representation $\ad\rho$ viewed as an $H$-representation is \begin{align*}
		\Hom(\rho|_H,\rho|_H)&=\Hom(\alpha\oplus \alpha',\alpha\oplus  \alpha')
		\cong\kfield \oplus \kfield \oplus \alpha^{-1}\alpha' \oplus \alpha'^{-1}\alpha.
	\end{align*} Since the action of $G$ swaps $\alpha$ and $\alpha'$, we have $\ad\rho\cong\kfield\oplus \eta\oplus \Ind^G_H({\alpha'}^{-1}\alpha)$ as a $G$-representation. Hence $$\ad^0\rho=\ad\rho/\kfield\cong\eta\oplus \Ind^G_H({\alpha'}^{-1}\alpha)$$ as desired.
\end{proof}

We now discuss the semisimple residual representation associated to a $p$-adic representation $\rho : G\to \GL_n(\overline{\Q}_p)$ of a profinite group $G$. It is known that the image $\rho(G)$ is contained in $\GL_n(T)$ for some finite extension $T/\Q_p$ and it preserves a lattice $\Lambda\subseteq T^n$ (see for example the proof of \cite[Cor 5]{Dickinson2001modularity} for an explanation). Hence $\rho$ is conjugate to some integral representation $$\rho_{\Lambda}:G\to \Aut(\Lambda)\cong \GL_n(\mathcal{O}_{T}).$$ Reduce $\rho_{\Lambda}$ modulo the maximal ideal of $\mathcal{O}_{T}$ and take its semisimplification. Then we obtain a semisimple residual representation $$\overline{\rho} :G\to \GL_n(\kfield_T).$$  It follows from Brauer--Nesbitt theorem that $\overline{\rho}$ does not depend on the choice of $\Lambda.$

Now suppose $n=2$. We have the following proposition from group theory \cite[Prop. 16]{Serre1971} that provides an exhaustive list of the possible image of an exceptional $\overline{\rho}$.
\begin{prop}\label{ExceptionalResidualImage}
	Suppose that the residual Galois representation $\overline{\rho}: G_\Q\to \GL_2(\overline{\F}_p)$ is exceptional. Then \begin{enumerate}
		\item [(i)] either the projective image of $\overline{\rho}$ is dihedral; or
		\item [(ii)] or the projective image of $\overline{\rho}$ is isomorphic to $A_4, S_4$ or $A_5.$
	\end{enumerate}  
\end{prop}

We define two characters $\epsilon_2$ and $\epsilon_2'$ that are later used to describe the Galois representation $\rho_f$ associated to an eigenform $f$ with $a_p(f)=0$. Let $K$ be the unramified quadratic extension of $\Q_p.$ By local class field theory, there is a unique homomorphism $K^\times\to G_K^\ab\to K^{\times}$ that sends $z\in \mathcal{O}_K^\times$ to $z$ and $p$ to $1.$ The two embeddings of $K$ into $\overline{\Q}_p$ then give rise to two characters $\epsilon_2$ and $\epsilon_2'$ from $G_K$ to $\overline{\Q}^\times$ that are permuted by the action of $\Gal(K/\Q_p)$. Denote by $\overline{\epsilon}_2$ and $\overline{\epsilon}_2'$ their reductions modulo $p$ respectively. We have $$\overline{\epsilon}_2'=\Frob(\overline{\epsilon}_2)=\overline{\epsilon}_2^p.$$

{\label{SEQModule}} Let $A$ be an abelian group sitting inside a short exact sequence of groups $$1\to A\to G\to H\to 1.$$ Then $A$ can be made into an $H$-module by defining  $h\cdot a :=\widetilde{h}a\widetilde{h}^{-1}$ for all $h\in H$ and $a\in A.$ 



\subsection{Locally induced Galois representations}{\label{LocallyInduced} In the remainder of this paper we assume $p$ to be an odd prime. 
	
We first explain the relation between the vanishing of the Fourier coefficient $a_p$ and locally induced Galois representations. 
Let $f=\sum_{n=1}^{\infty}a_nq^n$ be an eigenform in $S_k(\Gamma_1(1),\overline{\Q}_p)$ and denote by $\rho_f: G_\Q\to \GL_2(\overline{\Q}_p)$ the Galois representation attached to $f$. Then $\rho_f|_{\GQp}$ is crystalline of Hodge-Tate weights $(0,k-1)$ and the characteristic polynomial of the crystalline Frobenius is $X^2-a_p(f)X+p^{k-1}$ by \cite{Scholl1990motives}. 
When $v_p(a_p(f))>0$, the local representation $\rho_f|_{\GQp}$ is irreducible, and $a_p(f)$ completely determines it. In particular, 
when $a_p(f)=0$, we have the following nice description of $\rho_f|_{\GQp}$ by \cite[Prop 3.1.2]{breuil_2003II} (see also \cite[Thm 2.1.1]{CalegariSardari2021vanishing}).
\begin{thm}[Breuil]\label{Breuil} Suppose that $a_p(f)=0.$ Then
	$${\rho_f}|_{\GQp} \cong\Ind^{\GQp}_{G_K} \alpha , {\rm \quad  for\ }\alpha=\epsilon_2^{k-1} \otimes\psi|_{G_K}$$ where $\psi:\GQp\to K^\times$ is some unramified character  such that $\psi^2$ is the unramified character of $\GQp$ that sends the Frobenius element to $-1.$
\end{thm} 

We now work with the following setup.
	Let $\rho:G_\Q\to \GL_2(\overline{\Q}_p)$ be unramified outside $\{p,\infty\}.$ Suppose that locally at $p$, we have 
	$$\rho|_{\GQp}\cong\Ind^{\GQp}_{G_K}\alpha\quad{\rm and}\quad \rho|_{G_K}\cong \alpha\oplus\alpha'$$ for characters $\alpha,\alpha':G_K\to \overline{\Q}_p^{\times}.$ 
	Denote by $\overline{\alpha}$ and $\overline{\alpha}'$ the reductions of $\alpha$ and $\alpha'$ respectively. By Lemma \ref{InducedReducibleAndSemisimple}, $\overline{\rho}|_{\GQp}$ is semisimple and thus there are identifications $$\overline{\rho}|_{\GQp}\cong\Ind^{\GQp}_{G_K}\overline{\alpha}\quad {\rm and}\quad \overline{\rho}|_{G_K}\cong\overline{\alpha}\oplus\overline{\alpha}'.$$
	When we specialize $\rho$ to $\rho_f$ associated to an eigenform $f\in S_k(\Gamma_1(1),\overline{\Q}_p)$ with $a_p(f)=0$, we let $\alpha=\epsilon_2^{k-1} \otimes\psi|_{G_K}$ and $\alpha'={\epsilon}_2'^{k-1} \otimes\psi|_{G_K}$. Denote by  $\overline{\psi}$ the reduction of $\psi$ and recall from \S \ref{Prelim} that $\overline{\epsilon}_2'=\overline{\epsilon}_2^p.$ We then have $$\overline{\rho}_f|_{\GQp}\cong \Ind^{\GQp}_{G_K}\left(\overline{\epsilon}_2^{k-1}\otimes\overline{\psi}|_{G_K} \right)\quad{\rm and}\quad \overline{\rho}_f|_{G_K}\cong \left(\overline{\epsilon}_2^{k-1}\oplus\overline{\epsilon}_2^{p(k-1)}\right)\otimes\overline{\psi}|_{G_K}.$$ 

We deduce some properties of the induced local representation $\rho|_{\GQp}$.

\begin{lem}\label{AbsIrreducible}
	The local representation  $\overline{\rho}|_{\GQp}$ is absolutely irreducible. 
\end{lem}

\begin{proof}
	Let $\eta:\GQp\to \GQp/G_K\cong \{\pm 1\}\subset \overline{\F}_p^{\times}$ be the unramified quadratic character of $\GQp.$ Suppose that $\overline{\rho}|_{\GQp}\cong\Ind^{\GQp}_{G_K}\overline{\alpha}$ is reducible. By Lemma \ref{InducedReducibleAndSemisimple}, we have $\overline{\rho}|_{\GQp}$ splits into the direct sum $\overline{\alpha}\oplus (\overline{\alpha}\otimes \eta)$ and so ${\rm P}\overline{\rho}|_{\GQp}\cong\eta.$  
	Since $\overline{\rho}$ is unramified outside $\{p,\infty\},$ the projective representation ${\rm P}\overline{\rho}$ is now unramified at all finite places. The fixed field of ${\rm P}\overline{\rho}$ is then trivial because the narrow class number of $\Q$ is 1. This contradicts ${\rm P}\overline{\rho}|_{\GQp}\cong\eta.$ 
\end{proof}

\begin{lem}\label{InducedDihedral}
	The projective image ${\rm P}\overline{\rho}(\GQp)$ is isomorphic to a non-cyclic dihedral group $D_{2d}$ of order $2d$ and ${\rm P}\overline{\rho}(G_K)$ is isomorphic to a cyclic group of order $d.$ In the case $\rho=\rho_f,$ we have $d=\frac{p+1}{\gcd(k-1,p+1)};$ in particular, $d$ and $p+1$ have the same $2$-adic valuation.
\end{lem}

\begin{proof}
	By Lemma \ref{InducedProjectiveDihedral} we immediately deduce that ${\rm P}\overline{\rho}(\GQp)$ has non-cyclic dihedral image $D_{2d}$ and $\overline{\alpha}'^{-1}\overline{\alpha}$ is cyclic of order $d.$ Now assume that $\rho=\rho_f.$ Then $d$ is the order of $({\overline{\epsilon}_2'}^{-1}\overline{\epsilon}_2)^{k-1}=\overline{\epsilon}_2^{(p-1)(k-1)}$.  Given that $\overline{\epsilon}_2$ is cyclic of order $p^2-1,$ we have $$d=\frac{p^2-1}{\gcd((p-1)(k-1),p^2-1)}=\frac{p+1}{\gcd(k-1,p+1)}.$$ Since the weight $k$ of a level one modular form $f$ is even and $p$ is odd, $d$ and $p+1$ have the same $2$-adic valuation as claimed.
\end{proof}

\subsection{The fixed field of exceptional $\overline{\rho}$}\label{ExceptionalNumberField}
Let $\overline{\rho}$ be as before and let $L$ be the fixed field of  $P\overline{\rho}$
; then $L/\Q$ is a Galois extension unramified outside $\{p,\infty\}$ with Galois group $G\cong {\rm P}\overline{\rho}(G_\Q)$. 
The decomposition group $D$ of some (every) prime ideal of $L$ lying above $p$ is ${\rm P}\overline{\rho}(\GQp)$, which is isomorphic to $D_{2d}$ by Lemma \ref{InducedDihedral}. Since the inertia subgroup $I$ is a subgroup of ${\rm P}\overline{\rho}(G_K)\cong\Z/d\Z$ and $D/I$ is also cyclic, by the group structure of $D_{2d},$ we deduce that $I\cong {\rm P}\overline{\rho}(G_K)\cong \Z/d\Z.$

\begin{lem}\label{S4A5}
	Suppose that $\overline{\rho}$ is exceptional. Then the projective image $G\cong {\rm P}\overline{\rho}(G_\Q)$ is isomorphic to $S_4$ or $A_5$. 
\end{lem}

\begin{proof}
	We prove the lemma by showing other cases in the exhaustive list in Proposition \ref{ExceptionalResidualImage} cannot happen. 
	Suppose that $G$ is isomorphic to a dihedral group $D_{2n}$ of order $2n.$ Since $\Q(\sqrt{p^*})$ is the unique quadratic extension of $\Q$ unramified outside $\{p,\infty\},$ by Galois correspondence, $D_{2n}$ has a unique index two subgroup. Thus $n$ is odd. But since $\Q(\sqrt{p^*})$ is ramified at $p,$ we have $d=|I|$ is divisible by two and so $n$ is even, a contradiction.  Suppose that $G$ is isomorphic to $A_4.$ Since the only non-cyclic dihedral subgroup of $A_4$ is $D_4,$ we have $D\cong D_4$ and so $|I|=d=2.$ Since $D_4$ is normal in $A_4,$ the fixed field of $D$ is Galois over $\Q.$ It is ramified at $p,$ so $3$ divides $d,$ a contradiction. Thus $G$ is isomorphic to either $S_4$ or $A_5.$
\end{proof}

\begin{remark}\label{NotGloballyInduced}
	Note that if $\rho$ is globally induced, then by Lemma \ref{InducedProjectiveDihedral}, ${\rm P}\overline{\rho}(G_\Q)$ is isomorphic to a dihedral group, which is impossible by the proof of the lemma above. Thus $\rho$ is not globally induced.
\end{remark}

Assume from now on that $\overline{\rho}$ is exceptional.

\begin{lem}\label{GDI}
	\leavevmode
	\begin{enumerate}
		\item [(a)] If $G\cong S_4,$ then the decomposition group $D$ is embedded as either a non-normal $D_4$ with $I$ identified with the subgroup generated by a single transposition in $S_4,$ or $D_8$ with $I\cong\Z/4\Z$.
		\item [(b)] If $G\cong A_5,$ then the decomposition group $D$ is isomorphic to $D_4$ with $I\cong \Z/2\Z$, $D_6$ with $I\cong \Z/3\Z$ or $D_{10}$ with $I\cong\Z/5\Z.$ 
	\end{enumerate} 
\end{lem}

\begin{proof}
	When $G\cong S_4,$ the only non-cyclic dihedral subgroups $D_{2d}$ of $S_4$ are, up to conjugation, the normal $D_4$, the non-normal $D_4,$ $D_6$ and $D_8.$ If $D$ is of order four, we need to show that $I$ is generated by a single transposition. Pick any subfield $E\subseteq L$ of degree $[E:\Q]=3.$ (They are all conjugate.) Then $\Gal(L/E)$ is isomorphic to $D_8$, which necessarily contains all three double transpositions in $S_4$. If $I$ identified with the subgroup generated by a double transposition, then $I\subseteq \Gal(L/E)$ and thus $p$ is unramified in $E/\Q.$ This is impossible for the narrow class group of $\Q$ is trivial. Suppose that $D\cong D_6.$ Since the quadratic extension of $\Q$ contained in $L$ is ramified at $p,$ the inertia degree $|I|$ is even, contradicting $|I|=d=3.$ If $D$ is isomorphic to $D_8,$ the inertia subgroup is identified with the cyclic subgroup $\Z/4\Z$.
	
	When $G\cong A_5,$ up to conjugation the only non-cyclic dihedral subgroups $D_{2d}$ of $A_5$ are $D_4,D_6$ and $D_{10}.$ Thus the inertia subgroups are identified with $\Z/2\Z,\Z/3\Z$ and $\Z/5\Z$ respectively.  
\end{proof}

\subsection{Proof of Theorem \ref{Theorem}} 
Let $f = \sum_{n \ge 1} a_n q^n$ be a cuspidal eigenform with coefficients
in $E = \mathbf{Q}(a_1, a_2, ... )$ of level one and $a_p(f)=0$. By the main theorem of \cite{Hatada1979}, the eigenvalue $a_p(f)$ of the Hecke operator $T_p$ is congruent to $1 + p$ modulo $8$ when $p \neq 2$.
It follows that if $a_p(f)=0$ then $1+p\equiv0\pmod 8$ and thus $p\equiv 7 \pmod 8$. Suppose that for
some embedding $E\hookrightarrow \overline{\Q}_p,$ the $p$-adic representation $\rho_f$ associated to $f$ has exceptional image. We identify $f$ with its image under $S_k(\Gamma(1),E)\hookrightarrow S_k(\Gamma(1),\overline{\Q}_p).$ 
It follows from Lemma \ref{InducedDihedral} and Lemma \ref{GDI} that if $p\neq 2,$ then $v_2(p+1)=v_2(d)\le 2.$ Thus $p$ is not congruent to $7$ modulo $8$, a contradiction. When $p=2,$ it is proved in \cite[\S2.6]{CalegariSardari2021vanishing} that there is no level 1 eigenform $f$ with $a_2(f)=0,$ which completes the proof.

\section{A Galois Deformation Problem}
\subsection{Deformation functor associated to $\overline{\rho}$}\label{DeformationFunctor}
Let us summarize all assumptions we have made in \S\ref{LocallyInduced} and \S\ref{ExceptionalNumberField} and add a few more to be assumed throughout this section.
We use Mazur's definition of deformations in  \cite{mazur1989deforming} due to the convenience of carrying out explicit calculations. 
	\begin{assumption}\label{Assumption}
		Fix $p>2.$ Let $\rho:G_\Q\to \GL_2(T)$ be a Galois representation that is unramified outside $\{p,\infty\},$ where $T$ is a finite extension of $\Q_p$, and assume that $\rho(G_\Q)$ is contained in $\GL_2(\mathcal{O}_T).$ Locally at $p$, suppose that $\rho$ is induced from $G_K$ and takes the form $$\rho|_{\GQp}=\Ind^{\GQp}_{G_K}\alpha,\quad \rho|_{G_K}=\diag\{\alpha,\alpha'\}$$ for characters $\alpha,\alpha':G_K\to T^{\times}.$
		Suppose further that the semisimple residual representation $\overline{\rho}:G_\Q\to \GL_2(\kfield)$ is exceptional for $\kfield=\kfield_T$. As a result, the order of $\overline{\rho}(G_\Q)$ is prime to $p$ by Proposition \ref{ExceptionalResidualImage} and Lemma \ref{AbsIrreducible}. By extending $T$ if necessary, we assume that $T$ satisfies the following: \begin{itemize}
			\item the eigenvalues of every element in the image of $\overline{\rho}$ land in $\kfield$;
			\item $\mathcal{O}_K$ is a subring of $W(\kfield);$
			\item and the irreducible characters of $S_4$ and $A_5$ are all defined over $\kfield.$
		\end{itemize}
		We have the identifications $$\overline{\rho}|_{\GQp}=\Ind^{\GQp}_{G_K}\overline{\alpha}\quad{\rm and}\quad \overline{\rho}|_{G_K}=\diag\{\overline{\alpha},\overline{\alpha}'\}.$$ Since $\overline{\rho}|_{\GQp}$ is absolutely irreducible, by Lemma \ref{InducedReducibleAndSemisimple}, the two characters $\overline{\alpha}$ and
		$\overline{\alpha}'$ are distinct. 
	\end{assumption}
	
	We now define a deformation problem associated to $\overline{\rho}.$ Let $\mathscr{C}$ be the category of local Artinian $W(\kfield)$-algebras $(A,\mathfrak{m})$ with a fixed identification $A/\mathfrak{m}\cong \kfield$.
	
	\begin{definition}\label{DefinitionOfD}
	Denote by $\mathbf{D}$ the deformation functor associated to $\overline{\rho}$ satisfying the following deformation conditions. For $(A,\mathfrak{m})$ that is an object in $\mathscr{C},$ every deformation 
	$[\rho_A]\in \mathbf{D}(A)$ \begin{itemize}
		\item is of determinant $\chi:G_\Q\xrightarrow{\det \overline{\rho}}\kfield^{\times}\to W(\kfield)^{\times}$ where $\kfield^{\times}\to W(\kfield)^{\times}$ is the Teichmuller character,
		\item  is unramified outside $\{p,\infty\}$,
		\item  and has a representative $\rho_A$ such that $\rho_A|_{G_K}=\diag\{\beta,\beta’\}$ where $\beta$ and $\beta'$ are $G_K$-characters which are $G_K$-liftings (deformations) of $\overline{\alpha}$ and $\overline{\alpha}'$ respectively. 
	\end{itemize} 
	\end{definition}
	
	\begin{remark} Since $\overline{\alpha}$ and $\overline{\alpha}'$ are distinct, the last condition implies that every representations in $[\rho_A]$ takes the form $\diag\{\beta,\beta' \}$ when restricted to $G_K$ by \cite[Lemma 1.3.3]{CalegariSardari2021vanishing}. Hence, this is a deformation condition as shown in \cite[Lemma 2.3.2]{CalegariSardari2021vanishing}. It follows from this condition that $[\rho_A]$ is locally induced from $\beta$, or equivalently from $\beta'$. 
	Indeed, by Frobenius reciprocity, we have \begin{multline*}
		\Hom_{\GQp}(\Ind^{\GQp}_{G_K}\beta,\rho_A)\cong\Hom_{G_K}(\beta, \res^{\GQp}_{G_K}\rho_A)\\=\Hom_{G_K}(\beta,\beta\oplus\beta') = \End_{G_K}(\beta)\oplus \Hom_{G_K}(\beta,\beta').
	\end{multline*} If $\varphi$ is a $G_K$-homomorphism from $\beta$ to $\beta'$, then for every $a\in A,$ $$\beta'(g)\varphi(a) =g\varphi(a)= \varphi(ga) = \varphi(\beta(g)a)=\beta(g)\varphi(a).$$ But since $\beta$ and $\beta'$ reduce to distinct characters $\overline{\alpha}$ and $\overline{\alpha}'$ modulo $\mathfrak{m},$ there is some $g\in G_K$ such that $\beta(g)-\beta'(g)$ is a unit in $A,$ which implies $\varphi(a)=0$ for all $a\in A.$ Thus the homomorphism group $\Hom_{G_K}(\beta,\beta')$ is trivial and we have $$\Hom_{\GQp}(\Ind^{\GQp}_{G_K}\beta,\rho_A)\cong \End_{G_K}(\beta),$$ which clearly contains an isomorphism. 
	\end{remark} The functor $\mathbf{D}$ is pro-represented by a complete local Noetherian ring $R$ by \cite[Proposition 2.1]{mazur1989deforming}. 
	Note that the deformation $[\rho]$ represented by $\rho:G_\Q\to \GL_2(\mathcal{O}_T)$ satisfies the last two conditions. We can twist $[\rho]$ to meet all three conditions and get an $\mathcal{O}_T$-point of $\Spec(R).$ In fact, we have $(\det\rho)^{-1}\chi \in 1+\mathfrak{m}_T$ where $\mathfrak{m}_T$ is the maximal ideal of $T$. Since the function $\sqrt{1+x}:=1+ \frac{x}{2}-\frac{x^2}{8}+\cdots$ converges for $x\in \mathfrak{m}_T$ when $p>2$, there exists a character $\chi'$ whose square is $(\det\rho)^{-1}\chi.$ This $\chi'$ is unramified outside $\{p,\infty\}$ and its reduction modulo $\mathfrak{m}_T$ is trivial. Hence $[\rho\otimes \chi']$  is an $\mathcal{O}_T$-point on $\Spec(R)$. We aim to show that $\Spec(R)$ cannot contain such a point unless it has finite image by analyzing the tangent space of $\mathbf{D}$. 
	
	We now describe the tangent space $\mathbf{D}(\kfield[x]/(x^2))$ of the deformation functor $\mathbf{D}$ as a Selmer group $H^1_{\Sigma}(G_\Q,\ad^0\overline{\rho})$ defined below. Since the local representation $\overline{\rho}|_{\GQp}$ is induced, by Lemma \ref{AdjointRepDecompWhenInduced}, $\ad^0\overline{\rho}$ as a $\GQp$-representation splits into $\eta\oplus \Ind^{\GQp}_{G_K} ({\overline{\alpha}'}^{-1}\overline{\alpha}),$ where $\eta$ is the character $\eta:\GQp\to \GQp/G_K\cong\{\pm 1\}\subseteq\kfield^{\times}.$ 
	\begin{definition}{\label{ConditionSigma}}
			Let $\Sigma=\left\{\Sigma_v:v{\rm \ places\ of\ }\Q\right\}$ be the collection of local conditions for $\ad^0\overline{\rho}$ defined by \begin{itemize}
				\item $\Sigma_v=H^1(G_v/I_v,\ad^0\overline{\rho})=\ker(\res:H^1(G_v,\ad^0\overline{\rho})\to H^1(I_v,\ad^0\overline{\rho}))$ being the unramified condition at finite places $v\nmid p,$
				\item $\Sigma_v=H^1(\GQp,\eta)$ at $v=p$,
				\item no conditions at $v=\infty$.
			\end{itemize} Denote by $H_\Sigma^1(G_\Q,\ad^0\overline{\rho})$ the preimage of $\prod_v\Sigma_v$ under $\res:H^1(G_\Q,\ad^0\overline{\rho})\to \prod_v H^1(G_v,\ad^0\overline{\rho}).$ 
	\end{definition}
	
	\begin{prop}\label{TangentSpace}
		The tangent space of the deformation functor $\mathbf{D}$ is  $H^1_\Sigma(G_\Q,\ad^0\overline{\rho})$.
	\end{prop}
	
	\begin{proof}
		The Selmer group cut out by the first two conditions in Definition \ref{DefinitionOfD} is the preimage of $\prod_{v\nmid p}\Sigma_v$ under $H^1(G_\Q,\ad^0\overline{\rho})\to \prod_{v\nmid p}H^1(G_v,\ad^0\overline{\rho}).$ We are left to deal with the condition at $p.$ 
		
		Let $[\widetilde{\rho}]$ be a deformation in $\mathbf{D}(\kfield[x]/(x^2))$ represented by a lift $\widetilde{\rho}:G_\Q\to \GL_2(\kfield[x]/(x^2))$ such that $$\widetilde{\rho}|_{G_K}=\diag\{\beta,\beta'\}$$ where $\beta$ and $\beta'$ are lifts of $\overline{\alpha}$ and $\overline{\alpha}'$ respectively. Denote by $c:G_\Q\to \ad^0\overline{\rho}$ the cocycle determined by $\widetilde{\rho}$ via the equation $$\widetilde{\rho}=(I+x\cdot c)\overline{\rho}.$$
	Comparing the two formula of $\widetilde{\rho}|_{G_K}$, we conclude that $c$ restricted to $G_K$ is valued in the subrepresentation $\eta.$ The same argument shows that $c$ restricted to ${\GQp\setminus G_K}$ is valued in $\eta$ as well. Thus the cohomology class determined by $c$ lands in $H^1(\GQp,\eta).$ 
	Conversely, one checks that a cohomology class in $H^1(\GQp,\eta)$ determines a deformation $[\widetilde{\rho}]\in \mathbf{D}(\kfield[x]/(x^2))$ that admits a splitting when restricted to $G_K$. This gives the desired description of the tangent space.
\end{proof}

\subsection{Vanishing conditions} In this subsection we study when the tangent space $H^1_{\Sigma}(G_\Q,\ad^0\overline{\rho})$ of $\mathbf{D}$ vanishes. 

Continue with the notations in the previous subsection. By Galois theory, the decomposition group $D,$ when viewed as a subgroup of $G,$ fixes a number field $F$ with $\Gal(L/F)= D$ isomorphic to a dihedral group $D_{2d}$. It follows from Lemma \ref{InducedProjectiveDihedral} that $\overline{\rho}$ restricted to $G_F$ is induced from an index two subgroup 
of $G_F$. Note that the action of $g\in G_\Q$ on $\ad^0\overline{\rho}$ only depends on its projective image. Since the projective images of $G_F$ and $\GQp$ coincide, both being $D$, by abuse of notation, we say 
$\ad^0\overline{\rho}$ splits into $\eta\oplus \Ind^{\GQp}_{G_K}(\overline{\alpha}'^{-1}\overline{\alpha})$ as a $G_F$-representation as well. Denote the second summand by $W.$

We define a few more Selmer groups.
\begin{definition}\label{CohomologyCondition}
	Suppose that $E$ is a number field and let $X$ be a finite $G_E$-module. Denote by $w$ finite places of $E.$ 
	\begin{itemize}
		\item Let $H_\mathcal{L}^1(G_E,X)$ be the preimage of $\prod_{w\nmid p}H^1(G_w/I_w,X^{I_w})\times \prod_{w|p}\{0\}$ under the restriction map $\res: H^1(G_E,X)\to \prod_w H^1(G_w,X),$ i.e., the Selmer group consisting of cohomology classes that are unramified at all finite places and vanishing at places dividing $p;$ 
		\item let $H_\mathcal{M}^1(G_E,X)$ be the preimage of $\prod_{w}H^1(G_w/I_w,X^{I_w})$ under the restriction map $\res:H^1(G_E,X)\to \prod_w H^1(G_w,X),$ i.e., the Selmer group consisting of cohomology classes that are unramified at all finite places;
		\item and let $H_\mathcal{N}^1(G_E,X)$ be the preimage of $\prod_{w\nmid p}H^1(G_w/I_w,X^{I_w})$ under the restriction map $\res:H^1(G_E,X)\to \prod_{w\nmid p} H^1(G_w,X)$, i.e., the Selmer group consisting of cohomology classes that are unramified away from places dividing $p$ or $\infty.$ 
	\end{itemize}
\end{definition}

 Recall the tangent space $H^1_\Sigma(G_\Q,\ad^0\overline{\rho})$ in Definition \ref{ConditionSigma}. The cohomology classes $[c]$ in the tangent space are exactly those that are mapped into $H^1(\GQp,\eta)$ under $$\res: H^1_{\mathcal{N}}(G_\Q,\ad^0\overline{\rho})\to H^1(\GQp,\eta)\oplus H^1(\GQp,W).$$ It is not obvious how to compute these classes directly, or equivalently, to determine whether the deformation $[\widetilde{\rho}]\in \mathbf{D}(\kfield[x]/(x^2))$ corresponding to $[c]$ is locally induced or not. But if we assume that the residual representation $\overline{\rho}$ is globally induced, then one can hope to show a deformation of $\overline{\rho}$ to $\kfield[x]/(x^2)$ is locally induced if and only if it is globally induced, providing a criterion to determine whether the deformation is locally induced. So we restrict $\overline{\rho}$ to $G_F$ and study when classes in $H^1_{\mathcal{N}}(G_F,\ad^0\overline{\rho})$ are sent into $\prod_{w|p}H^1(G_w,\eta)$ under $$(\res_1,\res_2): H^1_{\mathcal{N}}(G_F,\ad^0\overline{\rho})\cong H^1_{\mathcal{N}}(G_F, \eta)\oplus H^1_{\mathcal{N}}(G_F,W)\to \prod_{w|p} H^1(G_w,\eta)\oplus \prod_{w|p} H^1(G_w,W).$$ Here, $w$ are places of $F$. We relate the two restriction maps $\res$ and $(\res_1,\res_2)$ by the following commutative diagram. We denote the vertical restriction map on the left by $\Psi$ and that on the right by $\Psi_{\loc}$.

 \begin{equation*}\label{Diagram}\tag{\textasteriskcentered}
\begin{tikzcd}
	H^1_{\mathcal{N}}(G_\Q,\ad^0\overline{\rho}_f)\arrow{rr}{\res} \arrow{d}[swap]{\Psi} && H^1(\GQp,\eta)\oplus H^1(\GQp,W) \arrow{d}{\Psi_\loc} \\
	H^1_{\mathcal{N}}(G_F, \eta)\oplus H^1_{\mathcal{N}}(G_F,W)\arrow{rr}{(\res_1,\res_2)}&&\prod_{w|p} H^1(G_w,\eta)\oplus \prod_{w|p} H^1(G_w,W).
\end{tikzcd}
\end{equation*} 

\begin{lem}\label{Main}
	In diagram (\ref{Diagram}), if \begin{enumerate}
		\item $\Psi$ does not map a nontrivial cohomology class into $H^1_{\mathcal{N}}(G_F,\eta)$, 
		\item and the restriction map $\res_2:H^1_{\mathcal{N}}(G_F,W)\to \prod_{w|p}H^1(G_w,W)$ is an injection,
	\end{enumerate} then the tangent space $H^1_\Sigma(G_\Q,\ad^0\overline{\rho})$ vanishes. 
\end{lem}

\begin{proof}
	 This follows from a diagram chasing argument. Let $[c]$ be a class in the tangent space $H^1_\Sigma(G_\Q,\ad^0\overline{\rho}),$ which by definition consists of classes that are mapped into $H^1(\GQp,\eta)$ under $\res$ in (\ref{Diagram}).  Then $$\Psi_\loc\circ\res([c])\in \Psi_\loc\left( H^1(\GQp,\eta)\right)\subseteq \prod_{w|p}H^1(G_w,\eta).$$ Since the diagram commutes, we have $$(\res_1,\res_2)\circ\Psi([c]) = \Psi_\loc\circ \res([c])\in \prod_{w|p}H^1(G_w,\eta).$$ Writing $\Psi([c])=[c_1]+[c_2]$ where $[c_1]\in H^1_{\mathcal{N}}(G_F,\eta)$ and $[c_2]\in H^1_{\mathcal{N}}(G_F,W)$, we have $$(\res_1,\res_2)\circ\Psi([c]) = \res_1([c_1])+\res_2([c_2])\in \prod_{w|p}H^1(G_w,\eta).$$ Thus $\res_2([c_2])=0.$ It follows from assumption (2) that  $[c_2]=0$. Hence, $\Psi([c])=[c_1]$ lies in $H^1_{\mathcal{N}}(G_F,\eta).$ By assumption (1), we conclude $[c]=0$, which completes the proof.
\end{proof}

We now study when the two assumptions in Lemma \ref{Main} hold. Assumption (1) in Lemma \ref{Main} holds automatically by the following lemma. 

\begin{lem}{\label{CohomologyBig}}
	The restriction map
	$$\Psi :H^1(G_\Q,\ad^0\overline{\rho})\to H^1(G_F,\eta)\oplus H^1(G_F,W)$$ maps a class $[c]\in H^1(G_\Q,\ad^0\overline{\rho})$ into $H^1(G_F,\eta)$ only if $[c]=0.$
\end{lem}

\begin{proof}
	Let $M$ be the fixed field of $\overline{\rho}.$ 
	We first show that $[c]\in H^1(G_\Q,\ad^0\overline{\rho})$ vanishes if $[c]|_{G_M}=0.$ 
	Since the order of $\overline{\rho}(G_\Q)\cong \Gal(M/\Q)$ is prime to $p,$ the cohomology group $H^1(\Gal(M/\Q),\ad^0\overline{\rho})$ vanishes. It follows from the inflation-restriction sequence that $H^1(G_\Q,\ad^0\overline{\rho})$ injects into $H^1(G_M,\ad^0\overline{\rho}).$ Thus if $[c]|_{G_M}=0,$ then $[c]=0$ in $H^1(G_\Q,\ad^0\overline{\rho}).$ 
	
	Now it suffices to prove $[c]|_{G_M}=0.$ The cochain $c$ representing $[c]$ determines a lift $\widetilde{\rho}$ of $\overline{\rho}$ by setting $\widetilde{\rho}= (1+x\cdot c) \overline{\rho}.$   We have the following short exact sequence 
	$$1\to \widetilde{\rho}(G_M)\to \widetilde{\rho}(G_\Q) \to \overline{\rho}(G_\Q)\to 1.$$ Note that  the image 
	$\widetilde{\rho}(G_M)$ is abelian because $$\widetilde{\rho}(G_M)=1+x\cdot c(G_M)\cong c(G_M)\subset 
	\ad^0\overline{\rho}.$$
	Following \S\ref{SEQModule}, the $\kfield$-vector space $V$ spanned by $c(G_M)$ is then a $G_\Q$-subrepresentation of $\ad^0\overline{\rho}.$ 
	If the image of $[c]$ under $H^1(G_\Q,\ad^0\overline{\rho})\to H^1(G_F,\ad^0\overline{\rho})$ lies in $H^1(G_F,\eta)$, then $c(G_M)\subseteq c(G_F)\subseteq \eta$ and thus $V\subseteq \eta.$  If $c(G_M)$ is nonzero, then $\eta=V$ is a $G_\Q$-representation. But $\eta$ is not a $G_\Q$-subrepresentation of $\ad^0\overline{\rho}$ by the following sublemma and the proof is complete. 
\end{proof}

\begin{sublem}\label{AdjointIrreducible}
	The $G_F$-representation $\eta$ is not a $G_\Q$-representation.
\end{sublem}
\begin{proof}
	Since $\overline{\rho}|_{G_K}=\diag\{\overline{\alpha},\overline{\alpha}'\},$ we have $\eta=\{\diag\{a,-a\}:a\in \kfield\}$. It follows from a straightforward calculation that in order to preserve $\eta,$ $\overline{\rho}(g)$ has to be either diagonal or anti-diagonal
	for all $g\in G_\Q$. Hence, $\overline{\rho}(g)^2$ is diagonal and it commutes with all matrices in $\overline{\rho}(G_K).$ Passing to $\PGL_2(\kfield),$ we have that for all $h\in G$, $h^2$ commutes with every element in $I\subset D.$ However, by Lemma 
	\ref{GDI}, one checks directly that this does not happen.
\end{proof}

We now study when assumption (2) in Lemma \ref{Main} holds.
Let $C=\Cl(L)/p\Cl(L)$ be the mod-$p$ class group of $L$. If $A/L$ is the maximal unramified abelian $p$-exponent extension of $L$, then class field theory gives an isomorphism $\Gal(A/L) \cong C$. By the uniqueness of $A$, $A/\Q$ is also Galois and we have the short exact sequence $$1\to C\to \Gal(A/\Q)\to G\to 1.$$ Following \S\ref{SEQModule}, the $\F_p$-vector space $C$ is endowed with a $G$-action.

Recall from Definition \ref{CohomologyCondition} of the Selmer conditions that the kernel of $\res_2$ is  $H^1_{\mathcal{L}}(G_F,W).$ Then assumption (2) amounts to saying that $H^1_{\mathcal{L}}(G_F,W)$ vanishes. We first embed the cohomology group $H^1_\mathcal{L}(G_F,W)$ into a homomorphism group. 

\begin{lem}\label{Hom}
	The Selmer group $H^1_{\mathcal{L}}(G_F,W)$ embeds into $\Hom_{D}(C,W).$
\end{lem}

\begin{proof}
	 Since $p$ does not divide the order of $\Gal(L/F)\le G,$ it follows from the inflation-restriction sequence that $H^1(G_F,W)$ embeds into $H^1(G_L,W)^{D}. $ This embedding restricts to $$H^1_{\mathcal{L}}(G_F,W)\to H^1_\mathcal{L}(G_L,W)^{D}.$$ The group $H^1_\mathcal{L}(G_L,W)^{D}$ is by definition a subgroup of $H^1_\mathcal{M}(G_L,W)^{D}.$ Since the action of $G_L$ on $\ad^0\overline{\rho}$ is trivial and $W$ is a $\Z/p\Z$-module, we have $$H^1_{\mathcal{M}}(G_L,W)^{D}=\Hom(C,W)^{D}=\Hom_{D}(C,W).$$ Hence, $H^1_{\mathcal{L}}(G_F,W)$ can be viewed as a subgroup of $\Hom_D(C,W).$ 
\end{proof}

To study the vanishing condition of $\Hom_D(C,W),$ we give a less restrictive definition of nicely exceptional semisimple residual representations. Fix an identification between $G$ and $S_4$ or $A_5.$ For a subgroup $H$ of $S_4$ or $A_5,$ by $L^H$ we mean the subfield of $L$ that is fixed by the subgroup whose image is $H$ under the identification $G\cong S_4$ or $G\cong A_5.$ If $H_1$ and $H_2$ are two subgroups of $G$ that are conjugate to each other, the fields $L^{H_1}$ and $L^{H_2}$ are isomorphic by Galois theory. In general, two isomorphic subgroups of $G$ do not need to be conjugate. But when $H$ is unique up to conjugation within its isomorphism class, we write $L^{[H]}$ for the isomorphism class of number fields $L^H$ where $[H]$ stands for the isomorphism class of the subgroup $H.$ For example, when $H$ is a subgroup of $S_4$ that is isomorphic to $S_3$, then $H$ is unique up to conjugation and we write $L^{S_3}$ for the isomorphism class of $L^H.$ 

\begin{definition}\label{NicelyExceptional}
	A semisimple Galois representation $\overline{\rho}: G_\Q\to \GL_2(\overline{\F}_p)$ is said to be nicely exceptional if it is exceptional and if $p$ is prime to the class numbers of the following fields:
	\begin{enumerate}
		\item $L^{S_3}$ and $L^{\Z/4\Z}$ when $G\cong S_4$;
		\item $L^{A_4}$ and $L^{\Z/5\Z}$ when $G\cong A_5$.
	\end{enumerate}
When $G$ is not $S_4$ or $A_5,$ the condition on class numbers is vacuously satisfied.
\end{definition}

\begin{lem}\label{HomVanishes}
	Suppose that $\overline{\rho}$ is nicely exceptional. Then $\Hom_{D}(C,W)=0.$ 
\end{lem}

\begin{proof}
	See \S{\ref{Proofof3.4.5}}. The proof only involves a tedious manipulation of representations of $G,D$ and $I.$
\end{proof}

Now we conclude the vanishing conditions for the tangent space. 
\begin{lem}\label{TangentSpaceVanishes}
	Suppose that $\overline{\rho}$ is nicely exceptional. Then the tangent space $H^1_{\Sigma}(G_\Q,\ad^0\overline{\rho}_f)$ is trivial.
\end{lem}

\begin{proof}
	It follows from Lemma \ref{Main}, Lemma \ref{CohomologyBig}, Lemma \ref{Hom} and Lemma \ref{HomVanishes}.
\end{proof}

\subsection{Proof of Theorem \ref{ThmFinite}}
If $\overline{\rho}$ is nicely exceptional, then $H^1_{\Sigma}(G_\Q,\ad^0\overline{\rho})$, which is the tangent space of $\mathbf{D}$, is trivial by Lemma \ref{TangentSpaceVanishes}. Thus, the deformation ring $R$ of $\mathbf{D}$ is a quotient of $W(\kfield).$ Since the order of $\overline{\rho}(G_\Q)$ is prime to $p,$ we can lift the finite group $\overline{\rho}(G_\Q)$ to $\GL_2(W(\kfield))$ and obtain a $W(\kfield)$-point $[\rho_1]$ of $\Spec(R)$. (We twist $\rho_1$ if necessary as in \S\ref{DeformationFunctor} so that $\det\rho_1=\chi$.) Hence $R$ is $W(\kfield).$ On the other hand, $[\rho\otimes \chi']$ is an $\mathcal{O}_T$-point of $\Spec(R)$ as explained in \S\ref{DeformationFunctor}. Since the only nontrivial homomorphisms from $W(\kfield)$ to $\mathcal{O}_T$ are embeddings,
$[\rho\otimes \chi']$ is also a $W(\kfield)$-point of $\Spec(R)$. By the uniqueness of the $W(\kfield)$-point on $\Spec(R)=\Spec(W(\kfield))$, we have $[\rho\otimes\chi']=[\rho_1]$ and so $\rho$ has finite image up to a twist.
\hfill\qed

\subsection{Proof of Lemma \ref{HomVanishes}} {\label{Proofof3.4.5}}
Continue with the notations in \S 3.2. The idea is to write  $$\Hom_D(C,W)\cong \Hom_{\kfield[D]}(C\otimes\kfield,W)\cong\Hom_{\kfield[G]}(C\otimes \kfield,\Ind^G_DW),$$ where the second isomorphism follows from Frobenius reciprocity. 
Note that in $\Hom_{D}(C,W),$ both $C$ and $W$ are viewed as $\Z/p\Z$-modules (or equivalently, $\F_p$-vector spaces) with $D$-action. By base changing to $\kfield,$ we work with representations over $\kfield$ to study when $\Hom_{\kfield[G]}(C\otimes \kfield,\Ind^G_DW)$ vanishes. Since the order of $G$ is prime to $p,$ it follows from  Maschke's theorem that finite dimensional representations of $G,D$ and $I$ over $\kfield$ are semisimple. 

We first describe $W$ as a $D$-representation over $\kfield$ and then we compute $\Ind^G_DW.$ In order to proceed, we need to name some irreducible characters of $D$ and $G$. We treat here the case $G\cong S_4.$ (When $G$ is isomorphic to $A_5,$ the proof is similar and actually simpler because all subgroups of $A_5$ are determined up to conjugation uniquely by their size.) Fix an identification $G\cong S_4.$ When the decomposition subgroup $D$ is identified with the non-normal $D_4$, we let 
$\rho_1$ and $\rho_2$ be the two irreducible characters of $D$ that are nontrivial on $I;$ when the decomposition subgroup $D$ is isomorphic to $D_8,$ we let 
$\chi^\Box$ be the unique two dimensional irreducible character of $D.$ Let $\sgn,\chi^{\perp},\sgn\otimes \chi^{\perp}$ and $\chi^{(5)}$ be the four nontrivial irreducible characters of $G\cong S_4.$ Here, the character $\chi^{\perp}$ is the orthogonal complement to the trivial character in the defining representation of $S_4$; the character $\sgn$ is the sign representation of $S_4;$ and the character $\chi^{(5)}$ is the unique two dimensional irreducible character of $S_4.$

\begin{sublem}{\label{InducedW2}}
	\leavevmode
	 \begin{enumerate}
		\item When $D\cong D_4$, we have $W\cong\rho_1\oplus\rho_2$ as a $D$-representation over $\kfield$ and $$\Ind^G_{D}W\cong \sgn\oplus\chi^{\perp}\oplus(\chi^{\perp}\otimes\sgn)^2\oplus \chi^{(5)}.$$
		\item When $D\cong D_8,$ we have $W\cong\chi^{\Box}$ as a $D$-representation over $\kfield$ and $$\Ind^G_{D}W\cong \chi^{\perp}\oplus (\chi^{\perp}\otimes\sgn).$$
	\end{enumerate} 
\end{sublem}

\begin{proof}
	Recall that 
	$W= \Ind^{\GQp}_{G_K}(\overline{\alpha}'^{-1}\overline{\alpha})$ as a $\GQp$-representation over $\kfield.$ 
	Regarding $W$ as a $D$-representation over $\kfield$, we have $W=\Ind^{D}_{I}(\overline{\alpha}'^{-1}\overline{\alpha}),$  where $\overline{\alpha}'^{-1}\overline{\alpha}$ is viewed as a character of $I$. 
	
	When $D$ is isomorphic to $D_4$, the order of $\overline{\alpha}'^{-1}\overline{\alpha}$ is $d=2.$ 
	We then have $$\overline{\alpha}'^{-1}\overline{\alpha}=\overline{\alpha}^{-1}\overline{\alpha}'\cong \sgn|_I$$ because $I$ is identified with the subgroup generated by a single transposition in $S_4.$ Hence, $W\cong\sgn|_{I}\oplus \sgn|_{I}$ as an $I$-representation over $\kfield$. 
	By semisimplicity of representations of $G$ and Frobenius reciprocity, 
	we have $$\Ind^G_{D}W=\Ind^G_D(\Ind^D_I(\sgn|_I))=\Ind^G_I(\sgn|_I)\cong\sgn\oplus\chi^{\perp}\oplus(\chi^{\perp}\otimes\sgn)^2\oplus \chi^{(5)}$$ as claimed.
	
	When $D$ is isomorphic to $D_8$, the order of $\overline{\alpha}'^{-1}\overline{\alpha}$ is $d=4$. 
	Then $\overline{\alpha}'^{-1}\overline{\alpha}$ is not equal to $\overline{\alpha}^{-1}\overline{\alpha}'$ and so $\overline{\alpha}'^{-1}\overline{\alpha}$ does not extend to a character of $D$. By Lemma \ref{InducedReducibleAndSemisimple}, $W$ is irreducible as a $D$-representation over $\kfield$. Hence it is the unique two dimensional irreducible representation $\chi^{\Box}$ of $D$. 
	By the same reasoning as before, we have $$\Ind^G_{D}W\cong\chi^{\perp}\oplus (\chi^{\perp}\otimes \sgn)$$ as claimed.
\end{proof}

We have the general results for an $S_4$ extension $L/\Q$ that is unramified outside $\{p,\infty\}.$

\begin{sublem}\label{SubfieldClassNumber} Let $L/\Q$ be an $S_4$ extension unramified outside $\{p,\infty\}$ and let $C$ be $\Cl(L)/p\Cl(L).$ Suppose that $p$ is prime to the order of $S_4.$ Then \begin{enumerate}
		\item [(a)]  $\Hom_{\kfield[G]}(C\otimes \kfield,\sgn)=0;$
		\item [(b)] $\Hom_{\kfield[G]}(C\otimes \kfield,\chi^\perp)=0$ if $p$ does not divide the class number of $L^{S_3}$;
		\item [(c)] $\Hom_{\kfield[G]}(C\otimes\kfield,\chi^\perp\otimes \sgn)=0$ if $p$ does not divide the class number of $L^{\Z/4\Z}$.
		\item [(d)] $\Hom_{\kfield[G]}(C\otimes\kfield,\chi^{(5)})=0$ if $p$ does not divide the class number of $L^{D_8}.$
	\end{enumerate}
\end{sublem}

\begin{proof}
	Note that for every representation $U$ of $G$ over $\kfield,$ we have $$\Hom_{\kfield[G]}(C\otimes \kfield,U)\cong \Hom_{\F_p[G]}(C,U)$$ where the latter $U$ is viewed as an $\F_p$-vector space.
	
	For part (a), if $\Hom_{\kfield[G]}(C\otimes \kfield,\sgn)\neq 0,$ then $\Hom_{\F_p[G]}(C,\sgn)\neq 0$. Thus $A_4$ acts trivially on a nontrivial subspace $V$ of $C.$ Viewing $C/V$ as a normal subgroup of $C,$ we have the short exact sequence $$1\to V\to \Gal(A^{C/V}/L^{A_4})\to A_4\to 1.$$ Since $p$ is prime to the order of $S_4,$ the sequence splits and by virtue of the fact that $A_4$ acts trivially on $V$ we have $\Gal(A^{C/V}/L^{A_4})\cong A_4\times V.$ By Galois theory, $(A^{C/V})^{A_4}$ is an unramified abelian $p$-exponent extension of the quadratic field $L^{A_4}=\Q(\sqrt{p^*})$. Hence $p$ divides the class number of $\Q(\sqrt{p^*})$. 
	However, when $p\equiv 1\pmod 4,$ the class number of $\Q(\sqrt{p})$ is smaller than $p$ (see for example \cite[\S5.6]{washington1997introduction}); when $p\equiv 3\pmod 4$  the class number of $\Q(\sqrt{-p})$ is smaller than $p$ (see for example \cite[\S1.13]{mazur1989deforming}).  
	
	For part (b), (c) and (d), we replace $A_4$ in the above argument by $S_3, \Z/4\Z$ and $D_8$ respectively. Then the homomorphism groups vanish if $p$ does not divide the class numbers of these fields.  
\end{proof}

\begin{proof}[Proof of Lemma \ref{HomVanishes}]
	Let $H_1\le H_2\le S_4$ be a chain of subgroups such that $H_1\cong \Z/4\Z$ and $H_2\cong D_8.$ Then $L^{H_2}$ is a subfield of $L^{H_1}.$ If the class number of $L^{H_1}$ is not divisible by $p,$ neither is that of $L^{H_2}.$ Thus $p$ does not divide the class numbers of all the number fields in the isomorphism class $L^{D_8}.$ 
	By Sublemma \ref{SubfieldClassNumber}, we then have \begin{align*}
		\Hom_{\kfield[G]}\left(C\otimes\kfield,\sgn\oplus \chi^{\perp}\oplus (\chi^\perp\otimes \sgn)^2\oplus \chi^{(5)}\right)=0.
	\end{align*} When $D\cong D_4,$ by Frobenius reciprocity and Sublemma \ref{InducedW2}, it follows that $$\Hom_D(C,W)\cong \Hom_{\kfield[D]}(C\otimes \kfield,W)\cong \Hom_{\kfield[G]}\left(C\otimes\kfield,\Ind_{D}^GW\right)=0.$$ Similar considerations apply to the case $D\cong D_8. $
\end{proof}

\subsection*{Acknowledgments} I would like to thank my PhD advisor Frank Calegari for suggesting the problem, for sharing his insights throughout the project and for providing extensive comments on an earlier version of the paper. I also thank the anonymous reviewer for their careful reading, as well as their helpful comments and corrections. This paper was funded in part by NSF Grant DMS-2001097.

\bibliographystyle{abbrv}
\bibliography{ref}

\end{document}